\documentclass[12pt]{amsart}
\usepackage{amsfonts,amssymb,amscd,amsmath,enumerate,verbatim,color}
\usepackage[latin1]{inputenc}
\usepackage{amscd}
\usepackage{latexsym}

\usepackage[dvipdfmx]{graphicx}
\usepackage{mathptmx}

\def\NZQ{\mathbb}               

\def\ZZ{{\NZQ Z}}
\def\RR{{\NZQ R}}

\def\frk{\mathfrak}               

\def\Phi{{\frk N}}
\def\eb{{\mathbf e}}

\def\xb{{\mathbf x}}
\def\yb{{\mathbf y}}

\def\opn#1#2{\def#1{\operatorname{#2}}} 
\opn\gr{gr}

\def\Ac{{\mathcal A}}
\def\Bc{{\mathcal B}}

\def\Hc{{\mathcal H}}
\def\Sc{{\mathcal S}}

\def\Gc{{\mathcal G}}
\def\Fc{{\mathcal F}}
\def\Oc{{\mathcal O}}
\def\Pc{{\mathcal P}}
\def\Qc{{\mathcal Q}}
\def\Rc{{\mathcal R}}

\def\Cc{{\mathcal C}}

\def\Vol{{\textnormal{Vol}}}

\newtheorem{Theorem}{Theorem}[section]
\newtheorem{Lemma}[Theorem]{Lemma}
\newtheorem{Corollary}[Theorem]{Corollary}
\newtheorem{Proposition}[Theorem]{Proposition}

\theoremstyle{definition}

\newtheorem{Example}[Theorem]{Example}

\newtheorem{Conjecture}[Theorem]{Conjecture}

\let\epsilon\varepsilon
\let\phi=\varphi
\let\kappa=\varkappa
\opn\dis{dis}
\opn\height{height}
\opn\dist{dist}
\def\pnt{{\raise0.5mm\hbox{\large\bf.}}}

\opn\Lex{Lex}
\opn\conv{conv}

\begin{document}

\title{The $h^*$-polynomials of locally anti-blocking lattice polytopes and their $\gamma$-positivity}
\author{Hidefumi Ohsugi and Akiyoshi Tsuchiya}
\address{Hidefumi Ohsugi,
	Department of Mathematical Sciences,
	School of Science and Technology,
	Kwansei Gakuin University,
	Sanda, Hyogo 669-1337, Japan} 
\email{ohsugi@kwansei.ac.jp}

\address{Akiyoshi Tsuchiya,
Graduate school of Mathematical Sciences,
University of Tokyo,
Komaba, Meguro-ku, Tokyo 153-8914, Japan} 
\email{akiyoshi@ms.u-tokyo.ac.jp}


\begin{abstract}
A lattice polytope $\mathcal{P} \subset \mathbb{R}^d$ is called a locally anti-blocking polytope  if for any closed orthant $\RR^d_{\varepsilon}$ in $\mathbb{R}^d$, $\mathcal{P} \cap \mathbb{R}^d_{\varepsilon}$ is unimodularly equivalent to an anti-blocking polytope by reflections of coordinate hyperplanes.
In the present paper, we give a formula for the $h^*$-polynomials of locally anti-blocking lattice polytopes.
In particular, we discuss the $\gamma$-positivity of the $h^*$-polynomials of locally anti-blocking reflexive polytopes.
\end{abstract}

\maketitle

\section*{Introduction}
A \textit{lattice polytope} is a convex polytope all of whose vertices have integer coordinates.
A lattice polytope $\Pc \subset \RR_{\geq 0}^d$ of dimension $d$ is called  \textit{anti-blocking} if for any $\yb=(y_1,\dots,y_d) \in \Pc$ and $\xb=(x_1,\dots,x_d) \in \RR^d$ with $0 \leq x_i \leq y_i$ for all $i$, it holds that $\xb \in \Pc$. 
Anti-blocking polytopes were introduced and studied by Fulkerson \cite{F1, F2}
in the context of combinatorial optimization.
See, e.g., \cite{Sch}.
For $\varepsilon \in \{-1,1\}^d$ and $\xb \in \RR^d$, set $\varepsilon \xb :=(\varepsilon_1 x_1,\ldots,\varepsilon_d x_d) \in \RR^d$.
Given an anti-blocking lattice polytope $\Pc \subset \RR_{\ge 0}^d$ of dimension $d$, we define 
\[
\Pc^{\pm}:=\{
\varepsilon \xb \in \RR^d :  \varepsilon \in \{-1,1\}^{d}, \ \xb \in \Pc
\}.
\]
Since  $\Pc$ is an anti-blocking lattice polytope, $\Pc^\pm$ is convex (and a lattice polytope).
Moreover, for any $\varepsilon \in \{-1,1\}^d$ and $\xb \in \Pc^{\pm}$, we have $\varepsilon \xb \in \Pc^{\pm}$. The polytope $\Pc^{\pm}$ is called an \textit{unconditional lattice polytope} (\cite{KOS}).
In general, $\Pc^{\pm}$ is symmetric with respect to all coordinate hyperplanes. In particular, the origin ${\bf 0}$ of $\RR^d$ is in the interior ${\rm int}(\Pc^{\pm})$.
Given $\varepsilon = (\varepsilon_1,\ldots, \varepsilon_d) \in \{-1,1\}^d$,
let $\RR^d_{\varepsilon}$ denote the closed orthant 
$\{ (x_1,\ldots, x_d) \in \RR^d  : x_i  \varepsilon_i \ge 0 \mbox{ for all } 1 \leq i \leq d\}$.
A lattice polytope $\Pc \subset \RR^d$ of dimension $d$ is called \textit{locally anti-blocking}
 (\cite{KOS}) if,
 for each $\varepsilon \in \{-1,1\}^d$ , there exists an anti-blocking lattice polytope $\Pc_{\varepsilon} \subset \RR_{\ge 0}^d$ of dimension $d$ such that $\Pc \cap \RR^d_{\varepsilon}=\Pc_{\varepsilon}^{\pm} \cap \RR^d_{\varepsilon}$.
Unconditional polytopes are locally anti-blocking.

In the present paper, we investigate the $h^*$-polynomials of locally anti-blocking lattice polytopes.
First, we give a formula for the $h^*$-polynomials of locally anti-blocking lattice polytopes in terms of that of unconditional lattice polytopes.
In fact,
\begin{Theorem}
\label{hpolypm}
Let $\Pc \subset \RR^d$ be a locally anti-blocking lattice polytope of dimension $d$ and for each $\varepsilon \in \{-1,1\}^d$, let
$\Pc_{\varepsilon}$ be an anti-blocking lattice polytope of dimension $d$ such that $\Pc \cap \RR^d_\varepsilon = \Pc_{\varepsilon}^\pm \cap 
\RR^d_\varepsilon$.
	Then the $h^*$-polynomial of $\Pc$ satisfies
$$
h^*(\Pc, x)
=  \frac{1}{2^{d}}\sum_{\varepsilon \in \{-1,1\}^d}
h^*(\Pc_{\varepsilon}^\pm, x).
$$
In particular, $h^*(\Pc,x)$ is $\gamma$-positive if $h^*(\Pc_{\varepsilon}^{\pm},x)$ is $\gamma$-positive for all $\varepsilon \in \{-1,1\}^d$.
\end{Theorem}

Second, we discuss the $\gamma$-positivity of the $h^*$-polynomials of locally anti-blocking reflexive polytopes. 
A lattice polytope is called \textit{reflexive} if the dual polytope is also a lattice polytope.
Many authors have studied reflexive polytopes from viewpoints of combinatorics, commutative algebra and algebraic geometry.
In \cite{hibi}, Hibi characterized reflexive polytopes in terms of their $h^*$-polynomials.
To be more precise, a lattice polytope of dimension $d$ is (unimodularly equivalent to) a reflexive polytope if and only if the $h^*$-polynomial is a palindromic polynomial of degree $d$.
On the other hand, in \cite{KOS}, locally anti-blocking reflexive polytopes were characterized.
In fact, a locally anti-blocking lattice polytope $\Pc \subset \RR^d$ of dimension $d$ is reflexive if and only if for each $\varepsilon \in \{-1,1\}^d$, there exists a perfect graph $G_{\varepsilon}$ on $[d]:=\{1,\ldots,d\}$ such that $\Pc \cap \RR^d_{\varepsilon}=\Qc_{G_{\varepsilon}}^{\pm} \cap \RR^d_{\varepsilon}$, where $\Qc_{G_{\varepsilon}}$ is the stable set polytope of $G_{\varepsilon}$.
Moreover, every locally anti-blocking reflexive polytope possesses a regular unimodular triangulation. 
This fact and the result of Bruns--R\"omer \cite{BR}
 imply that its $h^*$-polynomial is unimodal.

In the present paper, we discuss whether the $h^*$-polynomial of a locally anti-blocking reflexive polytope has a stronger property, which is called \textit{$\gamma$-positivity}.
In \cite{OTinterior}, a class of lattice polytopes $\Bc_G$ arising from finite simple graphs $G$ on $[d]$, which are called \textit{symmetric edge polytopes of type B}, was given.
Symmetric edge polytopes of type B are unconditional, and they are reflexive if and only if the underlying graphs are bipartite. Moreover, when they are reflexive, the $h^*$-polynomials are always $\gamma$-positive.
On the other hand, in \cite{ecp}, another family of lattice polytopes $\Cc^{(e)}_P$ arising from finite partially ordered sets $P$ on $[d]$, which are called \textit{enriched chain polytopes}, was given.
Enriched chain polytopes are unconditional and  reflexive, and their $h^*$-polynomials are always $\gamma$-positive.
Combining these facts and Theorem \ref{hpolypm}, we know that,
 for a locally anti-blocking reflexive polytope $\Pc$,
if every $\Pc \cap \RR_\varepsilon^d$ is the intersection of $\RR^d_{\varepsilon}$ and either an enriched chain polytope or a symmetric edge reflexive polytope of type B, then the $h^*$-polynomial of $\Pc$ is $\gamma$-positive (Corollary~\ref{average_cor}).
By using this result, we show that the $h^*$-polynomials of several classes of reflexive polytopes are $\gamma$-positive.

In Section \ref{sec:typeA}, we will discuss the $\gamma$-positivity of the $h^*$-polynomials of \textit{symmetric 
edge polytopes of type A}, which are reflexive polytopes arising from finite simple graphs.
In \cite{HJMsymmetric}, it was shown that the $h^*$-polynomials of the symmetric edge polytopes of type A of complete bipartite graphs are $\gamma$-positive.
We will show that for a large class of finite simple graphs, which includes complete bipartite graphs, the $h^*$-polynomials of the symmetric edge polytopes of type A are $\gamma$-positive (Subsection 4.1). 
Moreover, by giving explicit $h^*$-polynomials of del Pezzo polytopes and pseudo-del Pezzo polytopes, we will show that the $h^*$-polynomial of every pseudo-symmetric simplicial reflexive polytope is $\gamma$-positive (Theorem \ref{thm:pseudo-symmetric}).

In Section \ref{sec:chain}, we will discuss the $\gamma$-positivity of $h^*$-polynomials of \textit{twinned chain polytopes} $\Cc_{P,Q} \subset \RR^d$, which are reflexive polytopes arising from two finite partially ordered sets $P$ and $Q$ on $[d]$. In \cite{twineedchainpolytopes}, it was shown that twinned chain polytopes $\Cc_{P,Q}$ are locally anti-blocking and each $\Cc_{P,Q} \cap \RR_{\varepsilon}^d$ is the intersection of $\RR^d_{\varepsilon}$ and an enriched chain polytopes. Hence the $h^*$-polynomials of $\Cc_{P,Q} $ are $\gamma$-positive. 
We will give a formula for the $h^*$-polynomials of twinned chain polytopes in terms of the left peak polynomials of finite partially ordered sets  (Theorem \ref{thm:twinnedhpoly}).
Moreover, we will define \textit{enriched $(P,Q)$-partitions} of  $P$ and $Q$, and show that the Ehrhart polynomial of the twined chain polytope $\Cc_{P,Q}$ of $P$ and $Q$ coincides with a counting polynomial of enriched $(P,Q)$-partitions (Theorem \ref{thm:enrichedPQpart}).
 
This paper is organized as follows: In Section 1, we will review the theory of Ehrhart polynomials, $h^*$-polynomials, and reflexive polytopes.
In Section 2, we will introduce several classes of anti-blocking polytopes and unconditional polytopes.
In Section 3, we will investigate the $h^*$-polynomials of locally anti-blocking lattice polytopes. In particular, we will prove Theorem \ref{hpolypm}.
We will discuss symmetric edge polytope of type A in Section 4, and twinned chain polytopes in Section 5.

\subsection*{Acknowledgment}
The authors are grateful to the anonymous referees for their careful reading and helpful comments.
The authors were partially supported by JSPS KAKENHI 18H01134, 19K14505 and 19J00312.
 
\section{Ehrhart theory and Reflexive polytopes}
In this section, we review the theory of Ehrhart polynomials, $h^*$-polynomials, and reflexive polytopes.
Let $\Pc \subset \RR^d$ be a lattice polytope of dimension $d$.
Given a positive integer $m$, we define
$$L_{\Pc}(m)=|m \Pc \cap \ZZ^d|.$$
Ehrhart \cite{Ehrhart} proved that $L_{\Pc}(m)$ is a polynomial in $m$ of degree $d$ with the constant term $1$.
We say that $L_{\Pc}(m)$ is the \textit{Ehrhart polynomial} of $\Pc$.
The generating function of the lattice point enumerator, i.e., the formal power series
$$\text{Ehr}_\Pc(x)=1+\sum\limits_{k=1}^{\infty}L_{\Pc}(k)x^k$$
is called the \textit{Ehrhart series} of $\Pc$.
It is well known that it can be expressed as a rational function of the form
$$\text{Ehr}_\Pc(x)=\frac{h^*(\Pc,x)}{(1-x)^{d+1}}.$$
Then $h^*(\Pc,x)$ is a polynomial in $x$ of degree at most $d$ with nonnegative integer coefficients (\cite{Stanleynonnegative}) and it
is called
the \textit{$h^*$-polynomial} (or the \textit{$\delta$-polynomial}) of $\Pc$. 
Moreover, one has $\Vol(\Pc)=h^*(\Pc,1)$, where $\Vol(\Pc)$ is the normalized volume of $\Pc$.

A lattice polytope $\Pc \subset \RR^d$ of dimension $d$ is called \textit{reflexive} if the origin of $\RR^d$ is a unique lattice point belonging to the interior of $\Pc$ and its dual polytope
\[\Pc^\vee:=\{\yb \in \RR^d  :  \langle \xb,\yb \rangle \leq 1 \ \text{for all}\  \xb \in \Pc \}\]
is also a lattice polytope, where $\langle \xb,\yb \rangle$ is the usual inner product of $\RR^d$.
It is known that reflexive polytopes correspond to Gorenstein toric Fano varieties, and they are related to
mirror symmetry (see, e.g., \cite{mirror,Cox}).
In each dimension there exist only finitely many reflexive polytopes 
up to unimodular equivalence (\cite{Lag})
and all of them are known up to dimension $4$ (\cite{Kre}).
In \cite{hibi}, Hibi characterized reflexive polytopes in terms of their $h^*$-polynomials.
We recall that a polynomial $f \in \RR[x]$ of degree $d$ is said to be {\em palindromic} if $f(x)=x^df(x^{-1})$. Note that if a lattice polytope of dimension $d$ has interior lattice points, then the degree of its $h^*$-polynomial is equal to $d$.
\begin{Proposition}[\cite{hibi}]
	Let $\Pc \subset \RR^d$ be a lattice polytope of dimension $d$ with ${\bf 0} \in {\rm int}(\Pc)$. 	Then $\Pc$ is reflexive if and only if $h^*(\Pc,x)$ is a palindromic polynomial of degree $d$.
\end{Proposition}

Next, we review properties of polynomials. Let  $f= \sum_{i=0}^{d}a_i x^i$ be a polynomial with real coefficients and $a_d \neq 0$.
We now focus on the following properties.
\begin{itemize}
	\item[(RR)] We say that $f$ is {\em real-rooted} if all its roots are real.  
	\item[(LC)] We say that $f$ is {\em log-concave} if $a_i^2 \geq a_{i-1}a_{i+1}$ for all $i$.
	\item[(UN)] We say that $f$ is {\em unimodal} if $a_0 \leq a_1 \leq \cdots \leq a_k \geq \cdots \geq a_d$ for some $k$.
\end{itemize}  
If all its coefficients are nonnegative, then these properties satisfy the implications
\[
{\rm(RR)} \Rightarrow {\rm(LC)} \Rightarrow {\rm(UN)}.
\]
On the other hand,
the polynomial $f$ is {\em $\gamma$-positive} if $f$ is palindromic and there are $\gamma_0,\gamma_1,\ldots,\gamma_{\lfloor d/2\rfloor} \geq 0$ such that $f(x)=\sum_{i \geq 0}
\gamma_i \  x^i (1+x)^{d-2i}$.
The polynomial $\sum_{i \geq 0}\gamma_i \ x^i$ is called {\em $\gamma$-polynomial of $f$}. 
We can see that a $\gamma$-positive polynomial is real-rooted if and only if its $\gamma$-polynomial is real-rooted.
If $f$ is a palindromic and real-rooted, then it is $\gamma$-positive.
Moreover, if $f$ is $\gamma$-positive, then it is unimodal.
See, e.g., \cite{Athanasiadis, EulerianNumbers} for details.

For a given lattice polytope, a fundamental problem within the field of Ehrhart theory is to determine if its $h^*$-polynomial is unimodal.
One famous instance is given by reflexive polytopes that possess a regular unimodular triangulation.

\begin{Proposition}[\cite{BR}]
	Let $\Pc \subset \RR^d$ be a reflexive polytope of dimension $d$.
	If $P$ possesses a regular unimodular triangulation, then $h^*(\Pc, x)$ is unimodal.
\end{Proposition}

It is known that if a reflexive polytope possesses a flag regular unimodular triangulation all of whose maximal simplices contain the origin, then the $h^*$-polynomial coincides with the $h$-polynomial of a flag triangulation of a sphere (\cite{BR}). 
For the $h$-polynomial of a flag triangulation of a sphere, Gal (\cite{Gal}) conjectured the following:

\begin{Conjecture}[Gal Conjecture]
The $h$-polynomial of any flag triangulation of a sphere is $\gamma$-positive.
\end{Conjecture}

\section{Classes of anti-blocking polytopes and unconditional polytopes}
In this section, we introduce several classes of anti-blocking polytopes and unconditional polytopes.
Throughout this section, 
we associate each subset $F \subset [d]$ with a $(0,1)$-vector 
$\eb_F = \sum_{i \in F} \eb_i \in \RR^d$, 
where each $\eb_i$ is $i$th unit coordinate vector in $\RR^d$.

\subsection{$(0,1)$-polytopes arising from simplicial complexes}
Let $\Delta$ be a simplicial complex on the vertex set $[d]$.
Then $\Delta$ is a collection of subsets of $[d]$ with $\{i\} \in \Delta$ for all $i \in [d]$ such that
if $F \in \Delta$ and $F' \subset F$, then $F' \in \Delta$.
In particular $\emptyset \in \Delta$ and $\eb_{\emptyset}= {\bf 0}$.
Let $\Pc_\Delta$ denote the convex hull of 
$
\left\{ \eb_F \in \RR^d :  F \in \Delta \right\}
$.
The following is an important observation.

\begin{Proposition}
Let $\Pc \subset \RR_{\ge 0}^d$ be a (0,1)-polytope of dimension $d$.
Then $\Pc$ is anti-blocking if and only if there exists a
simplicial complex $\Delta$ on $[d]$ such that  $\Pc = \Pc_\Delta$.
\end{Proposition}

\subsection{Stable set polytopes}
Let $G$ be a finite simple graph on the vertex set $[d]$
and $E(G)$ the set of edges of $G$.
(A finite graph $G$ is called simple if $G$ possesses no loop and no multiple edge.)
A subset $W \subset [d]$ is called \textit{stable}  
if, for all $i$ and $j$ belonging to $W$ with $i \neq j$,
one has $\{i,j\} \notin E(G)$.
We remark that a stable set is often called an \textit{independent set}.
Let $S(G)$ denote the set of stable sets of $G$.
One has $\emptyset \in S(G)$ and $\{ i \} \in S(G)$
for each $i \in [d]$.
The \textit{stable set polytope} $\Qc_G \subset \RR^{d}$
of $G$ is the $(0, 1)$-polytope defined by
\[
\Qc_G:={\rm conv}(\{\eb_W \in \RR^d \, : \, W \in S(G) \}).
\]
Then one has $\dim \Qc_G = d$.
Since we can regard $S(G)$ as a simplicial complex on $[d]$, $\Qc_G$ is an anti-blocking polytope.

Locally anti-blocking reflexive polytopes are characterized by stable set polytopes.
A {\em clique} of $G$ is a subset $W \subset [d]$
which is a stable set of the complement graph $\overline{G}$ of $G$.
The {\em chromatic number} of $G$ is the smallest integer $t \geq 1$ for which
there exist stable sets $W_{1}, \ldots, W_{t}$ of $G$ with
$[d] = W_{1} \cup \cdots \cup W_{t}$.
A finite simple graph $G$ is said to be {\em perfect} 
if, 
for any induced subgraph $H$ of $G$
including $G$ itself,
the chromatic number of $H$ is
equal to the maximal cardinality of cliques of $H$.
See, e.g., \cite{graphbook} for details on graph theoretical terminologies.

\begin{Proposition}[\cite{KOS}]
\label{KOSstable}
Let $\Pc \subset \RR^d$ be a locally anti-blocking lattice polytope of dimension $d$.
Then $\Pc \subset \RR^d$ is reflexive if and only if,
 for each $\varepsilon \in \{-1,1\}^d$, there exists a perfect graph $G_{\varepsilon}$ on $[d]$ such that $\Pc \cap \RR^d_{\varepsilon}=\Qc_{G_{\varepsilon}}^{\pm} \cap \RR^d_{\varepsilon}$.
\end{Proposition}

\subsection{Chain polytopes and enriched chain polytopes}
Let $(P, <_P)$ be a partially ordered set (poset, for short) on $[d]$.
A subset $A$ of $[d]$ is called an \textit{antichain} of $P$ if all $i$ and $j$ belonging to $A$ with $i \neq j$ are incomparable in $P$. In particular, the empty set $\emptyset$ and each $1$-element subset $\{i\}$ are antichains of $P$.
Let $\Ac(P)$ denote the set of antichains of $P$.
In \cite{twoposetpolytopes}, Stanley introduced the \textit{chain polytope} $\Cc_P$ of $P$ defined by
\[
\Cc_P:={\rm conv}(\{ \eb_A \in \RR^d : A \in \Ac(P) \}).
\]
It is known that chain polytopes are stable set polytopes.
Indeed, let $G_P$ be the finite simple graph on $[d]$ such that $\{i,j\} \in E(G_P)$ if and only if $i <_P j$ or $j <_P i$.
We call $G_P$ the \textit{comparability graph} of $P$.
It then follows that $\Ac(P)=S(G_P)$.
Hence the chain polytope $\Cc_P$ is the stable set polytope $\Qc_{G_P}$.
Therefore, chain polytopes are anti-blocking polytopes.
We remark that any comparability graph is perfect.

On the other hand, the \textit{enriched chain polytope} $\Cc^{(e)}_P$ of $P$ is the unconditional lattice polytope defined by
\[
\Cc^{(e)}_P:=\Cc_P^{\pm}.
\]
In \cite{ecp},
it was shown that the Ehrhart polynomial of $\Cc^{(e)}_P$ coincides with a counting polynomial of left enriched $P$-partitions.
We assume that $P$ is naturally labeled. 
A map $f: P \rightarrow \ZZ \setminus \{0\}$ is called an 
{\em enriched $P$-partition} (\cite{StembridgeEnriched})
if, for all $x, y \in P$ with $x <_P y$, $f$ satisfies
\begin{itemize}
	\item[(i)]
	$|f(x)| \le |f(y)|$;
	\item[(ii)]
	$|f(x)| = |f(y)|  \ \Rightarrow \ f(y) > 0$.
\end{itemize}
A map $f: P \rightarrow \ZZ$ is called a {\em left enriched $P$-partition} (\cite{Petersen}) if, for all $x, y \in P$ with $x <_P y$, $f$ satisfies
\begin{itemize}
	\item[(i)]
	$|f(x)| \le |f(y)|$;
	\item[(ii)]
	$|f(x)| = |f(y)| \ \Rightarrow \ f(y) \ge 0$.
\end{itemize}
We denote $\Omega_P^{(\ell)}(m)$ the number of left enriched $P$-partitions $f : P \to \ZZ$ with $|f(x)| \leq m$ for any $x \in P$, which is called the \textit{left enriched order polynomial} of $P$.

\begin{Proposition}[\cite{ecp}]
	Let $P$ be a naturally labeled finite poset on $[d]$. 
	Then one has 
	\[
	L_{\Cc^{(e)}_P}(m) =\Omega_P^{(\ell)} (m).
	\]
\end{Proposition}

Given a linear extension $\pi = (\pi_1,\dots,\pi_d)$ of a finite poset $P$ on $[d]$,
 a {\em left peak} of $\pi$ is an index  $1 \le i \le d-1$ such that 
$\pi_{i-1} <\pi_i > \pi_{i+1} $, where we set $\pi_0 =0$.
Let  ${\rm pk}^{(\ell)}(\pi)$ denote the number of  left peaks of $\pi$.
Then the  {\em left peak polynomial} $W_{P}^{(\ell)} (x)$ of $P$ is defined by
$$
W_{P}^{(\ell)} (x) = 
\sum_{\pi \in {\mathcal L} (P)} x^{\ {\rm pk}^{(\ell)}(\pi)},
$$
where $\mathcal{L}(P)$ is the set of linear extensions of $P$.
\begin{Proposition}[\cite{ecp}]
	\label{prop:enrichedchainhpoly}
	Let $P$ be a naturally labeled finite poset on $[d]$. 
	Then the $h^*$-polynomial of $\Cc^{(e)}_P$ is
	$$
	h^*(\Cc^{(e)}_P, x)
	=  (x+1)^d \ W_{P}^{(\ell)}  \left(  \frac{4x}{(x+1)^2} \right).
	$$
	In particular, $h^*(\Cc^{(e)}_P, x)$ is $\gamma$-positive.
\end{Proposition}

Note that if $Q$ is a finite poset which is obtained from $P$ by reordering the label, then $\Cc^{(e)}_P$ and $\Cc^{(e)}_Q$ are unimodularly equivalent. Hence the $h^*$-polynomials of enriched chain polytopes are always $\gamma$-positive.
\subsection{Symmetric edge polytopes of type B}
Let $G$ be a finite simple graph on $[d]$. 
We set 
\[
B_G :={\rm conv} (\{ {\bf 0} , \eb_1,\ldots,\eb_d\} \cup \{\eb_i + \eb_j : \{i,j\} \in E(G) \}).
\]
Then $B_G = \Pc_\Delta$ where $\Delta$ is a simplicial complex on $[d]$ obtained by
regarding $G$ as a 1-dimensional simplicial complex.
The \textit{symmetric edge polytope of type B} of $G$ is the unconditional lattice polytope defined by
\[
\Bc_G:= B_G^{\pm}.
\]

\begin{Proposition}[\cite{OTinterior}]
	Let $G$ be a finite simple graph on $[d]$.
	Then $\Bc_G$ is reflexive if and only if $G$ is bipartite.
\end{Proposition}

A {\em hypergraph} is a pair $\Hc = (V, E)$, where $E=\{e_1,\ldots,e_n\}$ is a finite multiset of non-empty subsets of $V=\{v_1,\ldots,v_m\}$. 
Elements of $V$ are called vertices and the elements of $E$ are the  hyperedges.
Then we can associate $\Hc$ to a bipartite graph ${\rm Bip} \Hc$
with a bipartition $V \cup E$ such that $\{v_i, e_j\}$ is an edge of ${\rm Bip} \Hc$
if $v_i \in e_j$.
Assume that ${\rm Bip} \Hc$ is connected.
A {\em hypertree} in $\Hc$ is a function ${\bf f}: E \rightarrow \{0,1,\ldots\}$
such that there exists a spanning tree $\Gamma$ of ${\rm Bip} \Hc$ 
whose vertices have degree ${\bf f} (e) +1$ at each $e \in E$.
Then we say that $\Gamma$ induces ${\bf f}$.
Let $B_\Hc$ denote the set of all hypertrees in $\Hc$.
A hyperedge $e_j \in E$ is said to be {\em internally active}
with respect to the hypertree ${\bf f}$ if it is not possible to 
decrease ${\bf f}(e_j)$ by $1$ and increase ${\bf f}(e_{j'})$ ($j' < j$) by $1$
so that another hypertree results.
We call a hyperedge {\em internally inactive} with respect to a hypertree
if it is not internally active and denote the number of such hyperedges 
of ${\bf f}$ by $\overline{\iota} ({\bf f}) $.
Then the {\em interior polynomial} of $\Hc$
is the generating function 
$I_\Hc (x) = \sum_{{\bf f} \in B_\Hc} x^{ \overline{\iota} ({\bf f}) }$.
It is known \cite[Proposition~6.1]{interior} that $\deg I_\Hc (x) \le \min\{|V|,|E|\} - 1$.
If $G = {\rm Bip} \Hc$, then we set $I_G (x) = I_\Hc (x)$.

Assume that $G$ is a bipartite graph with a bipartition
$V_1 \cup V_2 =[d]$.
Then let $\widetilde{G}$ be a connected bipartite graph on $[d+2]$
whose edge set is 
\[
E(\widetilde{G}) = E(G) \cup \{ \{i, d+1\}  : i \in V_1\} \cup \{ \{j, d+2\}  : j \in V_2 \cup \{d+1\}\}.
\]

\begin{Proposition}[\cite{OTinterior}]
	\label{hpolymain}
	Let $G$ be a bipartite graph on $[d]$.
	Then $h^*$-polynomial of the reflexive polytope $\Bc_G$ is
	$$
	h^*(\Bc_G, x)
	=  (x+1)^d I_{\widetilde{G}} \left(  \frac{4x}{(x+1)^2} \right).
	$$
	In particular, $h^*(\Bc_G, x)$ is $\gamma$-positive.
\end{Proposition}

\section{$h^*$-polynomials of locally anti-blocking lattice polytopes}

In the present section, we prove Theorem~\ref{hpolypm}, that is, a formula for the $h^*$-polynomials of locally anti-blocking lattice polytopes in terms of that of unconditional lattice polytopes.
Given a subset $J=\{j_1,\dots, j_r\}$ of $[d]$, let 
$$\pi_J : \RR^d \rightarrow \RR^r, \ \pi_J((x_1,\dots,x_d)) = (x_{j_1},\dots,x_{j_r})$$
denote the projection map.
(Here $\pi_\emptyset$ is the zero map.)

\begin{Proposition}
\label{hpolyformula}
	Let $\Pc \subset \RR_{\ge 0}^d$ be an anti-blocking lattice polytope.
Then we have
	\begin{eqnarray*}
	h^*(\Pc^\pm, x)
	&=&
	\sum_{j=0}^d \ \ 
	2^j (x-1)^{d-j}  
\sum_{J \subset [d], \ |J| = j} h^*(\pi_J (\Pc), x).
	\end{eqnarray*}
\end{Proposition}

\begin{proof}
The proof is similar to the discussion in \cite[Proof of Proposition 3.1]{OTinterior}.
The intersection of 
	$\Pc^\pm \cap \RR_\varepsilon^d$ and $\Pc^\pm \cap \RR_{\varepsilon'}^d$
	is of dimension $d-1$ if and only if $\varepsilon - \varepsilon' \in \{\pm 2 \eb_1, \ldots, \pm 2 \eb_d\}$.
Moreover, if $\varepsilon - \varepsilon' = 2 \eb_k$, then 
	$$
	(\Pc^\pm \cap \RR_\varepsilon^d) \cap (\Pc^\pm \cap \RR_{\varepsilon'}^d)
	= 
	\Pc^\pm \cap \RR_\varepsilon^d \cap\RR_{\varepsilon'}^d
	\simeq
	\pi_{[d] \setminus \{k\}} (\Pc^\pm) \cap \RR_{ \pi_{[d] \setminus \{k\}} (\varepsilon)}^{d-1}
   \simeq
   \pi_{[d] \setminus \{k\}} (\Pc)
   .
	$$
	Hence the Ehrhart polynomial $L_{\Pc^\pm}(m)$ satisfies the following:
	$$
	L_{\Pc^\pm}(m) = \sum_{j=0}^d \ \ 
	2^j (-1)^{d-j}  
\sum_{J \subset [d], \ |J| = j} L_{\pi_J (\Pc)}(m).
	$$
	Thus the Ehrhart series satisfies
	\begin{eqnarray*}
		\frac{h^*(\Pc^\pm, x)}{(1-x)^{d+1}}
		&=&
		\sum_{j=0}^d \ \ 
		2^j (-1)^{d-j}  
\sum_{J \subset [d], \ |J| = j} 
		\frac{h^*(\pi_J (\Pc), x)}{(1-x)^{j+1}},
	\end{eqnarray*}
	as desired.
\end{proof}

We now prove Theorem~\ref{hpolypm}.

\begin{proof}[Proof of Theorem~\ref{hpolypm}]
Given $J = \{j_1,\dots, j_r\} \subset [d]$ and $\varepsilon \in \{-1,1\}^{r}$, let
$$\RR_{J, \varepsilon}^d =
\{ \xb = (x_1,\ldots, x_d) \in \RR^d  : 
\pi_J(\xb) \in \RR_\varepsilon^r
\mbox{ and } x_j =0  \mbox{ for all } j \notin J  \}.$$
It then follows that $\Pc \cap\RR_{J, \varepsilon}^d$ is equal to 
$\pi_J(\Pc_{\varepsilon'})^\pm \cap \RR_\varepsilon^r$,
where $\pi_J(\varepsilon') = \varepsilon$.
Note that, given $J = \{j_1,\dots, j_r\} \subset [d]$ and $\varepsilon \in \{-1,1\}^{r}$,
we have
$
| \{ \varepsilon' \in \{-1,1\}^d :  \pi_J(\varepsilon') =  \varepsilon \} |
= 2^{d-r}.
$
Thus 
\begin{eqnarray*}
h^*(\Pc, x)
&=& 
\sum_{j=0}^d (x-1)^{d-j} 
\sum_{J \subset [d],\ |J| = j } \ \ 
\sum_{\varepsilon \in \{-1,1\}^j}
h^*(\Pc \cap\RR_{J, \varepsilon}^d, x)\\
&=&
\sum_{j=0}^d (x-1)^{d-j} 
\sum_{\varepsilon \in \{-1,1\}^d}\ \ 
\sum_{J \subset [d],\ |J| = j }
\frac{1}{2^{d-j}}
h^*(\pi_J(\Pc_{\varepsilon}), x)\\
&=&
\frac{1}{2^{d}}
\sum_{\varepsilon \in \{-1,1\}^d}\ \ 
\sum_{j=0}^d 2^j (x-1)^{d-j} 
\sum_{J \subset [d],\ |J| = j }
h^*(\pi_J(\Pc_{\varepsilon}), x)\\
&=&
\frac{1}{2^{d}}
\sum_{\varepsilon \in \{-1,1\}^d}\ \ 
h^*(\Pc_\varepsilon^\pm, x)
\end{eqnarray*}
by  Proposition~\ref{hpolyformula}.
\end{proof}

Combining Theorem \ref{hpolypm} and Propositions \ref{prop:enrichedchainhpoly} and \ref{hpolymain}, we have the following.

\begin{Corollary}
\label{average_cor}
Let $\Pc \subset \RR^d$ be a locally anti-blocking reflexive polytope.
If every $\Pc \cap \RR_\varepsilon^d$ is the intersection of $\RR_\varepsilon^d$ and either an enriched chain polytope or a symmetric edge reflexive polytope of type B, then the $h^*$-polynomial of $\Pc$ is $\gamma$-positive.
\end{Corollary}

Finally, we conjecture the following:
\begin{Conjecture}
\label{gammapositiveconjecture}
	The $h^*$-polynomial of any locally anti-blocking reflexive polytope is $\gamma$-positive.
\end{Conjecture}

Thanks to Theorem~\ref{hpolypm} and Proposition~\ref{KOSstable},
in order to prove Conjecture~\ref{gammapositiveconjecture},
it is enough to study unconditional lattice polytopes $\Qc_G^\pm$
where $\Qc_G$ is the stable set polytope of a perfect graph $G$.

\section{Symmetric edge polytopes of type A}
\label{sec:typeA}

Let $G$ be a finite simple graph on the vertex set $[d]$ and the edge set $E(G)$.
The {\em symmetric edge polytope} ${\mathcal A}_G \subset \RR^d$ of type A
is the convex hull of the set
$$
A(G) = \{ \pm( \eb_i - \eb_j ) \in \RR^d: \{ i, j\} \in E(G)  \}.
$$
The polytope ${\mathcal A}_G$ is introduced in \cite{fivemen, CSC} 
and called a ``symmetric edge polytope of $G$.''

\begin{Example}
\label{ex:cross}
Let $G$ be a tree on $[d]$.
Then $\Ac_G$ is unimodularly equivalent to a $(d-1)$-dimensional cross
polytope.
Hence we have $h^*(\Ac_G,x)= (x+1)^{d-1}$.
\end{Example}

It is known \cite[Proposition 4.1]{fivemen} that the dimension of $\Ac_G$ is $d-1$ if and only if $G$ is connected.
Higashitani \cite{Higashi} proved that ${\mathcal A}_G$ is simple if and only if
${\mathcal A}_G$ is smooth Fano if and only if
$G$ contains no even cycles.
It is known \cite{fivemen, CSC} that ${\mathcal A}_G$ is unimodularly equivalent to a reflexive polytope 
having a regular unimodular triangulation.
In particular,  $h^*$-polynomial of ${\mathcal A}_{G}$ is palindromic and unimodal.
For a complete bipartite graph $K_{\ell, m}$, it is known \cite{HJMsymmetric} that the
$h^*$-polynomial of ${\mathcal A}_{K_{\ell, m}}$ is real-rooted and hence $\gamma$-positive.


\subsection{Recursive formulas for $h^*$-polynomials}

In this section, we give several recursive formulas of $h^*$-polynomials of $\Ac_G$
when $G$ belongs to certain classes of graphs.
By the following fact, we may assume that $G$ is 2-connected if needed.

\begin{Proposition}
\label{2connected}
 Let $G$ be a graph and let
 $G_1,\ldots, G_s$ be $2$-connected components of $G$.
Then the $h^*$-polynomial of ${\mathcal A}_{G}$ satisfies
	\[
	h^*(\Ac_G,x)=h^*(\Ac_{G_1},x) \cdots h^*(\Ac_{G_s},x).
	\]
\end{Proposition}

\begin{proof}
Since $\Ac_G$ is the free sum of reflexive polytopes
$\Ac_{G_1}, \ldots, \Ac_{G_s}$,
a desired conclusion follows from \cite[Theorem 1]{Braun}.
\end{proof}

The {\em suspension} $\widehat{G}$ of a graph $G$ is the graph
on the vertex set $[d+1]$ and the edge set
$$ E(G)\cup
\{
\{i, d+1\} : i \in [d]
\} 
.$$
We now study the $h^*$-polynomial of ${\mathcal A}_{\widehat{G}}$.
Given a subset $S \subset [d]$,
$$
E_S := \{ e \in E(G) : |e \cap S| =1 \}
$$
is called a {\it cut} of $G$.
For example, we have $E_\emptyset = E_{[d]} = \emptyset$.
In general, it follows that $E_S = E_{[d] \setminus S}$.
We identify $E_S$ with the subgraph of $G$ on the vertex set $[d]$
and the edge set $E_S$.
By definition, $E_S$ is a bipartite graph.
Let ${\rm Cut}(G)$ be the set of all cuts of $G$.
Note that $|{\rm Cut}(G)| = 2^{d-1}$.
From Theorem~\ref{hpolypm} and  Proposition~\ref{hpolymain},
we have the following.

\begin{Theorem}
\label{thm:Asuspension}
	Let $G$ be a finite graph on $[d]$.
	Then ${\mathcal A}_{\widehat{G}}$ is unimodularly equivalent to a locally anti-blocking 
reflexive polytope whose $h^*$-polynomial is
$$
h^*({\mathcal A}_{\widehat{G}}, x)
=  \frac{1}{2^{d-1}}\sum_{H \in {\rm Cut}(G)}
h^*({\mathcal B}_H, x)
=  (x+1)^d f_G \left(  \frac{4x}{(x+1)^2} \right)
,
	$$
where 
$$
f_G (x)= \frac{1}{2^{d-1}} \sum_{H \in {\rm Cut}(G)} I_{\widetilde{H}} (x)
.$$
	In particular, $	h^*({\mathcal A}_{\widehat{G}}, x)$ is $\gamma$-positive.
	Moreover, $	h^*({\mathcal A}_{\widehat{G}}, x)$ is real-rooted if and only if $f_G(x)$ is real-rooted.
\end{Theorem}

\begin{proof}
Let ${\mathcal P} \subset \RR^d$ be the convex hull of
$$
\{ \pm \eb_1, \dots, \pm \eb_d\}
\cup
 \{ \pm( \eb_i - \eb_j ): \{ i, j\} \in E(G)  \}.
$$
Then ${\mathcal A}_{\widehat{G}}$ is lattice isomorphic to $\Pc$. 
Given $\varepsilon = (\varepsilon_1,\ldots, \varepsilon_d) \in \{-1,1\}^d$,
let $S_\varepsilon = \{ i \in [d] : \varepsilon_i = 1\}$.
Then ${\mathcal P} \cap \RR_\varepsilon^d$
is the convex hull of 
$$
\{{\bf 0}\} \cup 
\{ \varepsilon_i \eb_i : i \in [d] \}
\cup
 \{ \eb_i - \eb_j : \{ i, j\} \in E_{S_\varepsilon} ,  i \in S_\varepsilon\}.
$$
Hence  ${\mathcal P} \cap \RR_\varepsilon^d = \Bc_{E_{S_\varepsilon}} \cap \RR_\varepsilon^d$.
Thus $\Pc$ is a locally anti-blocking polytope and 
$$h^*({\mathcal A}_{\widehat{G}}, x)
=
\frac{1}{2^{d-1}}
\sum_{H \in {\rm Cut}(G)}h^*({\mathcal B}_H, x)
$$
by Theorem~\ref{hpolypm}.
\end{proof}

Let $G$ be a graph and let $e=\{i,j\}$ be an edge of $G$.
Then the graph $G/e$ obtained by the procedure 
\begin{itemize}
\item[(i)]
Delete $e$ and identify the vertices $i$ and $j$;
\item[(ii)]
Delete the multiple edges that may be created while (i)
\end{itemize}
is called the graph obtained from $G$ by {\em contracting} the edge $e$. 
Next, 
we will show that,
for any bipartite graph $G$ and $e \in E(G)$,
 $h^*({\mathcal A}_G, x)$ is $\gamma$-positive
if and only if so is $h^*(\Ac_{G/e}, x)$.
In order to show this fact, we need the theory of Gr\"obner bases of toric ideals.
Given a graph $G$ on the vertex set $[d]$ and the edge set $E(G)=\{e_1 ,\dots, e_n\}$,
let 
$$\Rc=K[t_1, t_1^{-1}, \dots, t_d, t_d^{-1} ,s ]$$
be the Laurent polynomial ring over a field $K$
and let 
$$\Sc=K[x_1, \dots, x_n, y_1, \dots, y_n, z]$$
be the polynomial ring over $K$.
We define the ring homomorphism $ \pi : \Sc \rightarrow \Rc$ by setting $\pi(z) = s$, 
$\pi(x_k) = t_i t_j^{-1} s$ and $\pi(y_k) = t_i^{-1} t_j s$
if $e_k = \{i,j\} \in E(G)$ and $i < j$.
The {\em toric ideal} $I_{\Ac_G}$ of $\Ac_G$ is the kernel of $\pi$.
(See, e.g., \cite{binomialideals} for details on toric ideals and Gr\"obner bases.)
We now define the notation given in \cite{HJMsymmetric}.
For any oriented edge $e_i$,
let $p_i$ denote the corresponding variable, i.e. 
$p_i = x_i$ or $p_i = y_i$ depending on the orientation
and let $\{p_i, q_i\} = \{x_i, y_i\}$.
Let $\Gc(G)$ be the set of all binomials $f$ satisfying one of the following:
\begin{equation}
\label{even}
f= \prod_{e_i \in I} p_i - \prod_{e_i \in C \setminus I} q_i,
\end{equation}
where $C$ is an even cycle in $G$ of length $2k$ with a fixed orientation,
and $I$ is a $k$-subset of $C$
such that $e_\ell \notin I$ for $\ell = \min\{i :  e_i \in C\}$;
\begin{equation}
\label{odd}
f= \prod_{e_i \in I} p_i - z\prod_{e_i \in C \setminus I} q_i,
\end{equation}
where $C$ is an odd cycle in $G$ of length $2k+1$ and $I$ is a $(k+1)$-subset of $C$;
\begin{equation}
\label{pmzero}
f = x_i y_i - z^2,
\end{equation}
where $1 \leq i \leq n$.
Then $\Gc(G)$ is a Gr\"obner basis of $I_{\Ac_G}$ with respect to
 a reverse lexicographic order $<$ induced by the ordering
$z < x_1 < y_1 < \dots < x_n < y_n$
 (\cite[Proposition 3.8]{HJMsymmetric}).
Here the initial monomial of each binomial is the first monomial.
Using this Gr\"obner basis, we have the following.

\begin{Proposition}
\label{contraction}
Let $G$ be a bipartite graph on $[d]$ and let $e \in E(G)$.
Then we have 
$$h^*(\Ac_G, x) = (x+1) h^*(\Ac_{G/e}, x).$$
\end{Proposition}

\begin{proof}
Let  $E(G)=\{e_1 ,\dots, e_n\}$ with $e = e_1 = \{i,j\}$.
Since $G$ is a bipartite graph, the Gr\"obner basis $\Gc(G)$ above consists of 
the binomials of the form (\ref{even}) and (\ref{pmzero}).

Since $G$ has no triangles, the procedure (ii) does not occur 
when we contract $e$ of $G$.
Hence $E(G/e) = \{e_2' , \dots, e_n' \}$ where $e_k'$ is obtained
from $e_k$ by identifying $i$ with $j$.
Let $G'$ be a graph obtained by adding an edge $e_1' = \{d+1, d+2\}$ to the graph $G/e$.
Then $\Gc(G')$ consists of all binomials $f$ satisfying one of the following:
\begin{equation}
f= \prod_{e_i \in I} p_i - \prod_{e_i \in C \setminus I} q_i,
\end{equation}
where $C$ is an even cycle in $G$ of length $2k$ with a fixed orientation and $e_1 \notin C$,
and $I$ is a $k$-subset of $C$
such that $e_\ell \notin I$ for $\ell = \min\{i :  e_i \in C\}$;
\begin{equation}
f= \prod_{e_i \in I} p_i - z\prod_{e_i \in C \setminus I} q_i,
\end{equation}
where $C \cup \{e_1\}$ is an even cycle in $G$ of length $2k+2$ and $I$ is a $(k+1)$-subset of $C$;
\begin{equation}
f = x_i y_i - z^2,
\end{equation}
where $1 \leq i \leq n$.
Hence $\{ {\rm in}_< (f)  :  f \in \Gc(G) \} = \{ {\rm in}_< (f)  :  f \in \Gc(G') \}$. 
By a similar argument as in the proof of \cite[Theorem 3.1]{HTperfect}, it follows that
$$
h^*(\Ac_G, x) =  h^*( \Ac_{G'}, x) =  h^*(\Ac_{\{e_1'\}}, x) h^*(\Ac_{G/e}, x)= (x+1) h^*(\Ac_{G/e}, x)
,$$
as desired.
\end{proof}

From Theorem \ref{thm:Asuspension},
Propositions~\ref{2connected} and \ref{contraction}
we have the following immediately.

\begin{Corollary}
\label{petapeta}
Let $G$ be a bipartite graph on $[d]$.
Then we have the following{\rm :}
\begin{itemize}
\item[(a)]
The $h^*$-polynomial $h^*(\Ac_{\widetilde{G}}, x) = (x+1) h^*(\Ac_{\widehat{G}}, x)$
is $\gamma$-positive.

\item[(b)]
If $G$ is obtained by gluing bipartite graphs $G_1$ and $G_2$ along with an edge $e$,
then
\begin{eqnarray*}
h^*(\Ac_G, x) &=&(x+1) h^*(\Ac_{G/e}, x)  \\
&=&(x+1) h^*(\Ac_{G_1/e}, x) h^*(\Ac_{G_2/e}, x)\\
&=&
h^*(\Ac_{G_1}, x) h^*(\Ac_{G_2}, x)/(x+1).
\end{eqnarray*}
\end{itemize}
\end{Corollary}

\smallskip

\noindent
\textbf{Remark.}
Corollary \ref{petapeta} (b) was recently generalized in \cite[Theorem~4.17]{DDM}.

\subsection{Pseudo-symmetric simplicial reflexive polytopes}
A lattice polytope $\Pc \subset \RR^d$ is called \textit{pseudo-symmetric} if there exists a facet $\Fc$ of $\Pc$ such that $-\Fc$ is also a facet of $\Pc$. 
Nill \cite{BNill} proved that 
any pseudo-symmetric simplicial reflexive polytope
$\Pc$ is a free sum of $\Pc_1, \dots, \Pc_s$,
where each $\Pc_i$ is one of the following:
\begin{itemize}
\item
cross polytope;

\item
del Pezzo polytope $V_{2m} = {\rm conv} ( \pm \eb_1,\dots,\pm \eb_{2m}, 
\pm( \eb_1 + \dots + \eb_{2m}) )$;

\item
pseudo-del Pezzo polytope $\widetilde{V}_{2m} = {\rm conv} ( \pm \eb_1,\dots,\pm \eb_{2m}, 
-\eb_1 - \dots - \eb_{2m} )$.
\end{itemize}
Note that a del Pezzo polytope is unimodularly equivalent to $\Ac_{C_{2m+1}}$
where $C_{2m+1}$ is an odd cycle of length $2m+1$ (see \cite{Higashi}).
The $h^*$-polynomial of $\Ac_{C_d}$ was essentially studied in the following papers
(see also the OEIS sequence A204621):
\begin{itemize}
\item
Conway--Sloane \cite[p.2379]{CS} computed $h^*(\Ac_{C_d},x)$ for small $d$ by 
using results of O'Keeffe \cite{O'Ke}
and gave a conjecture on the $\gamma$-polynomial of 
$h^*(\Ac_{C_d},x)$ (coincides with the $\gamma$-polynomial in Proposition~\ref{cycle h} below).
\item
General formulas for the coefficients of $h^*(\Ac_{C_d},x)$ were given by
Ohsugi--Shibata \cite{OhsugiShibata} and Wang--Yu \cite{WY}.
\end{itemize}
In order to give the $h^*$-polynomial of $\widetilde{V}_{2m}$,
we need the following lemma.

\begin{Lemma}
\label{deleteoneedge}
Let $G$ be a connected graph.
Suppose that an edge $e=\{i,j\} $ of $G$ is not a bridge.
Let $\Pc_e$ be the convex hull of $A(G) \setminus \{ \eb_i - \eb_j   \}$.
Then we have
$$
h^*(\Pc_e,x) = \frac{1}{2} ( h^*(\Ac_G,x) + h^*(\Ac_{G \setminus e},x) )
,$$
where $G \setminus e$ is the graph obtained by deleting $e$ from $G$.
\end{Lemma}

\begin{proof}
Note that $\Ac_{G \setminus e} \subset \Pc_e \subset \Ac_G$.
Since $G$ is connected and $e$ is not a bridge of $G$,
the dimension of each of $\Ac_G$ and $\Ac_{G \setminus e}$ is $d-1$.
Let $\Pc_e'$ denote the convex hull of $A(G) \setminus \{ - \eb_i + \eb_j   \}$,
which is unimodularly equivalent to $\Pc_e$.
Then $\Ac_G$ and $\Pc_e$ are decomposed into the following disjoint union:
\begin{eqnarray*}
\Ac_G &=& \Ac_{G \setminus e} \cup (\Pc_e \setminus \Ac_{G \setminus e}) 
 \cup (\Pc_e' \setminus \Ac_{G \setminus e}),\\
\Pc_e &=& \Ac_{G \setminus e} \cup (\Pc_e \setminus \Ac_{G \setminus e}) .
\end{eqnarray*}
Since $\Pc_e \setminus \Ac_{G \setminus e}$ is unimodularly equivalent to 
$\Pc_e' \setminus \Ac_{G \setminus e}$, we have a desired conclusion.
\end{proof}

The $h^*$-polynomials of $V_{2m}$ and $\widetilde{V}_{2m}$ are as follows:

\begin{Proposition}
\label{cycle h}
Let $C_d$ denote a cycle of length $d \ge 3$ and let $1 \le m \in \ZZ$.
Then we have
\begin{eqnarray*}
h^*(\Ac_{C_d},x) &=& \sum_{i=0}^{\lfloor \frac{d-1}{2} \rfloor}
\binom{2i}{i} x^i (x+1)^{d-2i-1},\\
h^*(V_{2m},x) &=& \sum_{i=0}^m
\binom{2i}{i} x^i (x+1)^{2m-2i},\\
h^*(\widetilde{V}_{2m},x) &=& (x+1)^{2m}
+ \sum_{i=1}^m
\binom{2i-1}{i-1} x^i (x+1)^{2m-2i}.
\end{eqnarray*}
In particular, the $h^*$-polynomials of $\Ac_{C_d}$, $V_{2m}$ and $\widetilde{V}_{2m}$ are $\gamma$-positive.
\end{Proposition}

\begin{proof}
The proof for $C_d$ is induction on $d$.
First, we have
$h^*(\Ac_{C_{3}},x) = x^2 + 4x +1 = (x+1)^2 + \binom{2}{1} x.$
If $d \ge 4$ is even, then
$$
h^*(\Ac_{C_{d}},x) =  (x + 1) h^*(\Ac_{C_{d-1}},x)
=\sum_{i=0}^{\frac{d-2}{2}}
\binom{2i}{i} x^i (x+1)^{d-2i-1}
=
\sum_{i=0}^{\lfloor \frac{d-1}{2} \rfloor}
\binom{2i}{i} x^i (x+1)^{d-2i-1}.
$$
Moreover, if $d = 2m +1$ ($2 \le m \in \ZZ$), then
the coefficient of $x^m$ in
$$
\sum_{i=0}^{\frac{d-1}{2}}
\binom{2i}{i} x^i (x+1)^{d-2i-1}
= (x+1) h^*(\Ac_{C_{d-1}},x)
+ \binom{2m}{m} x^m
$$
is 
$
\sum_{i=0}^m \binom{2i}{i} \binom{2m-2i}{m-i}
= 4^m = 2^{d-1}
$
and other coefficient is arising from $(x+1) h^*(\Ac_{C_{d-1}},x)$.
By a recursive formula in \cite[Theorem~2.3]{OhsugiShibata}, we have
$$h^*(\Ac_{C_d},x)= \sum_{i=0}^{\frac{d-1}{2}}
\binom{2i}{i} x^i (x+1)^{d-2i-1}.$$
Since $V_{2m}$ is unimodularly equivalent to $\Ac_{C_{2m+1}}$,
we have $h^*(V_{2m},x) = h^*(\Ac_{C_{2m+1}},x)$.
By Lemma~\ref{deleteoneedge}, it follows that
\begin{eqnarray*}
h^*(\widetilde{V}_{2m},x) 
&= &
\frac{1}{2} ( h^*(\Ac_{C_{2m+1}},x) +  h^*(\Ac_{P_{2m+1}},x) )\\
&=&
\frac{1}{2} \left( \sum_{i=0}^m
\binom{2i}{i} x^i (x+1)^{2m-2i} +  (x+1)^{2m} \right)\\
&=&(x+1)^{2m}
+ \sum_{i=1}^m
\binom{2i-1}{i-1} x^i (x+1)^{2m-2i}.
\end{eqnarray*}
\end{proof}

Thus it turns out that any pseudo-symmetric simplicial reflexive polytope
is a free sum of reflexive polytopes whose $h^*$-polynomial are $\gamma$-positive.
By \cite[Theorem 1]{Braun}, we have the following.

\begin{Theorem}
	\label{thm:pseudo-symmetric}
The $h^*$-polynomial of any pseudo-symmetric simplicial reflexive polytope is $\gamma$-positive.
\end{Theorem}

\begin{proof}
From results by Nill \cite{BNill}, any pseudo-symmetric simplicial reflexive 
polytope is a free sum of cross polytopes, del Pezzo polytopes and 
pseudo-del Pezzo polytopes. On the other hand, by \cite[Theorem 1]{Braun},
the $h^*$-polynomial of a free sum of reflexive polytopes 
$\Pc_1,\ldots,\Pc_s$ is equal to the product of their $h^*$-polynomials 
of $\Pc_1,\ldots, \Pc_s$. Hence by Example \ref{ex:cross} and Proposition~\ref{cycle h},
it follows that the $h^*$-polynomial of any pseudo symmetric 
simplicial reflexive polytope is $\gamma$-positive.
\end{proof}

\subsection{Classes of graphs such that $h^*({\mathcal A}_G, x)$ is $\gamma$-positive}

Using results in the present section, for example, 
$h^*({\mathcal A}_G, x)$ is $\gamma$-positive
if one of the following holds:
\begin{itemize}
\item
$G = \widehat{H}$ for some graph $H$ (e.g., $G$ is a complete graph, a wheel graph);
\item
$G = \widetilde{H}$ for some bipartite graph $H$ (e.g., $G$ is a complete bipartite graph);
\item
$G$ is a cycle;
\item
$G$ is an outerplanar bipartite graph.
\end{itemize}
Moreover, we can compute $h^*({\mathcal A}_G, x)$ explicitly in some cases.
We give examples of such calculations for known formulas 
(for complete graphs \cite{Ardila}, and for complete bipartite graphs \cite{HJMsymmetric}).

\begin{Example}[\cite{Ardila}]
By Theorem \ref{thm:Asuspension}, we have
$$h^*({\mathcal A}_{K_d}, x) = h^*({\mathcal A}_{\widehat{K_{d-1}}}, x)
=\frac{(x+1)^{d-1} }{2^{d-2}} \sum_{H \in {\rm Cut}(K_{d-1})} I_{\widetilde{H}} \left(\frac{4x}{(x+1)^2}\right)
.$$
If the edge set of $H \in {\rm Cut}(K_{d-1})$ is $E_S$ with $S \subset [d-1]$,
then $H$ is a complete bipartite graph $K_{|S|, d-1-|S|}$
and $I_{\widetilde{H}}(x) = \sum_{i\ge0} \binom{|S|}{i} \binom{d-|S|-1}{i} x^i$.
(Here $K_{0,d-1}$ denotes an empty graph.)
It then follows that
\begin{eqnarray*}
 h^*({\mathcal A}_{K_d}, x)
 &=& \frac{1}{2^{d-1}} 
\sum_{k=0}^{d-1}
\binom{d-1}{k} 
\sum_{i = 0}^{\lfloor \frac{d-1}{2} \rfloor} 4^i \binom{k}{i} \binom{d-k-1}{i} x^i (x+1)^{d-1-2i} 
\\
&=&
\frac{1}{2^{d-1}} 
\sum_{i = 0}^{\lfloor \frac{d-1}{2} \rfloor}
4^i x^i (x+1)^{d-1-2i} 
\sum_{k=i}^{d-i-1}
\binom{d-1}{k} 
\binom{k}{i} \binom{d-k-1}{i}\\
&=&
\frac{1}{2^{d-1}} 
\sum_{i = 0}^{\lfloor \frac{d-1}{2} \rfloor}
4^i x^i (x+1)^{d-1-2i} 
\sum_{k=i}^{d-i-1}
\binom{d-1}{2i}
\binom{2i}{i}
\binom{d-1-2i}{k-i}\\
&=&
\frac{1}{2^{d-1}} 
\sum_{i = 0}^{\lfloor \frac{d-1}{2} \rfloor}
4^i x^i (x+1)^{d-1-2i} 
\left(2^{d-1-2i} 
\binom{d-1}{2i}
\binom{2i}{i}\right)\\
&=&
\sum_{i = 0}^{\lfloor \frac{d-1}{2} \rfloor}
\binom{d-1}{2i}
\binom{2i}{i}
x^i (x+1)^{d-1-2i} .
\end{eqnarray*}

\end{Example}

\begin{Example}[\cite{HJMsymmetric}]
Let $G=K_{m,n}$. Then $\widetilde{G} = K_{m+1,n+1}$ and 
$$
h^*({\mathcal A}_{K_{m+1,n+1}}, x)
= (x+1) h^*({\mathcal A}_{\widehat{K_{m,n}}}, x)
=\frac{(x+1)^{m+n+1} }{2^{m+n-1}} \sum_{H \in {\rm Cut}(K_{m,n})} I_{\widetilde{H}} \left(\frac{4x}{(x+1)^2}\right).
$$
Let $V_1 \cup V_2$ be the partition of the vertex set of $K_{m,n}$, where $|V_1|=m$ and $|V_2|=n$.
If the edge set of  $H \in {\rm Cut}(K_{m,n})$ is $E_S$ with $S \subset [m+n]$,
then $H$ is the disjoint union of two complete bipartite graphs
$K_{k, \ell}$ and $K_{m-k,n-\ell}$,
and hence
$$I_{\widetilde{H}}(x) =
\left( \sum_{i\ge0} \binom{k}{i} \binom{\ell}{i} x^i \right)
\left( \sum_{j\ge0} \binom{m-k}{j} \binom{n-\ell}{j} x^j \right),
$$
where $k=|V_1 \cap S|$ and $\ell = n- |V_2 \cap S|$.
It then follows that
\begin{eqnarray*}
& &
h^*({\mathcal A}_{K_{m+1,n+1}}, x)\\
&=& \frac{x+1}{2^{m+n}} 
\sum_{k=0}^m
\sum_{\ell=0}^n
\binom{m}{k} \binom{n}{\ell}
\left(\sum_{i=0}^{\min(k,\ell)} 4^i \binom{k}{i} \binom{\ell}{i} x^i (x+1)^{k+\ell-2i} \right)
\\
& & \hspace{3cm}
\left(\sum_{j=0}^{\min(m-k,n-\ell)} 4^j \binom{m-k}{j} \binom{n-\ell}{j} x^j (x+1)^{m+n-k-\ell-2j} \right)
\\
&=&
\frac{1}{2^{m+n}} 
\sum_{i,j\ge0}
4^{i+j} 
x^{i+j} (x+1)^{n+m-2(i+j)+1} 
\sum_{k=i}^{m-j}
\binom{m}{k}
\binom{k}{i}
\binom{m-k}{j}
\sum_{\ell=i}^{n-j}
 \binom{n}{\ell}
\binom{\ell}{i}
 \binom{n-\ell}{j}.
\end{eqnarray*}
Since 
$$
\sum_{k=i}^{m-j}
\binom{m}{k}
\binom{k}{i}
\binom{m-k}{j}
=
\sum_{k=i}^{m-j}
\binom{m}{i+j}
\binom{i+j}{i}
\binom{m-(i+j)}{k-i}
=
2^{m-(i+j)} 
\binom{m}{i+j}
\binom{i+j}{i},
$$
we have
\begin{eqnarray*}
h^*({\mathcal A}_{K_{m+1,n+1}}, x) &=&
\sum_{i\ge0}
\sum_{j\ge0}
\binom{i+j}{i}^2
 \binom{m}{i+j} \binom{n}{i+j} x^{i+j }(x+1)^{m+n-2(i+j)+1} \\
& = & 
\sum_{\alpha =0}^{\min(m, n)}
\sum_{i=0}^\alpha
\binom{\alpha}{i}^2
 \binom{m}{\alpha} \binom{n}{\alpha} x^\alpha(x+1)^{m+n-2\alpha+1} \\
& = & 
\sum_{\alpha =0}^{\min(m, n)}
\binom{2 \alpha}{\alpha}
 \binom{m}{\alpha} \binom{n}{\alpha} x^\alpha(x+1)^{m+n-2\alpha+1}.
\end{eqnarray*}

\end{Example}

Finally, we conjecture the following:
\begin{Conjecture}
	The $h^*$-polynomial of any symmetric edge polytope of type A is $\gamma$-positive.
\end{Conjecture}

\section{Twinned chain polytopes}
\label{sec:chain}
In this section, we will apply Theorem \ref{hpolypm} to twinned chain polytopes.
For two lattice polytopes $\Pc, \Qc \subset \RR^d$, we set
\[
\Gamma(\Pc, \Qc):={\rm conv}(\Pc \cup (- \Qc)) \subset \RR^d.
\]
Let $P$ and $Q$ be two finite posets on $[d]$.
The \textit{twinned chain polytope} of $P$ and $Q$ is the lattice polytope defined by
\[
\Cc_{P,Q}:=\Gamma(\Cc_P,\Cc_Q).
\]
Then $\Cc_{P,Q}$ is reflexive.
Moreover, $\Cc_{P,Q}$ has a flag, regular unimodular triangulation all of whose maximal simplices contain the origin (\cite[Proposition 1.2]{HMTgamma}).
Hence we obtain the following:
\begin{Corollary}
	Let $P$ and $Q$ be two finite posets on $[d]$.
	Then the $h^*$-polynomial of $\Cc_{P,Q}$ coincides with the $h$-polynomial of a flag triangulation of a sphere.
\end{Corollary}
In {\cite[Proposition 2.2]{twineedchainpolytopes}} it was shown that $\Cc_{P,Q}$ is locally anti-blocking.
In general, for two finite posets $(P, <_P)$ and $(Q,<_Q)$ with $P \cap Q = \emptyset$, the \textit{ordinal sum} of $P$ and $Q$ is the  poset $(P \oplus Q, <_{P \oplus Q})$ on $P \oplus Q= P \cup Q$ such that $i <_{P \oplus Q} j$ if and only if (a) $i,j \in P$ and $i <_P j$, or (b) $i,j \in Q$ and $i <_Q j$, or (c) $i \in P$ and $j \in Q$.
Given a subset $I$ of $[d]$, we define the \textit{induced subposet} of $P$ on $I$ to be the finite poset $(P_I,<_{P_I})$ on $I$ such that $i <_{P_I} j$ if and only if $i <_P j$.
For $I \subset [d]$, let $\overline{I}:=[d] \setminus I$.
\begin{Proposition}[{\cite[Proposition 2.2]{twineedchainpolytopes}}]
	Let $P$ and $Q$ be two finite posets on $[d]$.
	Then for each $\varepsilon \in \{-1,1\}^d$, it follows that
	\[
	\Cc_{P,Q} \cap \RR^d_{\varepsilon}=\Cc^{\pm}_{P_{I_\varepsilon} \oplus Q_{\overline{I_{\varepsilon}}}} \cap \RR^d_{\varepsilon},
	\]
	where $I_{\varepsilon}=\{i \in [d] : \varepsilon_i=1 \}$.
\end{Proposition}
From this result, Theorem \ref{hpolypm} and Proposition \ref{prop:enrichedchainhpoly} we obtain the following:
\begin{Theorem}
	\label{thm:twinnedhpoly}
	Let $P$ and $Q$ be two finite posets on $[d]$.
	Then one has
	\[
	h^*(\Cc_{P,Q}, x)
	=  \frac{1}{2^{d}}\sum_{\varepsilon \in \{-1,1\}^d}
h^*(\Cc^{(e)}_{R_{\varepsilon}},x)=(x+1)^d f_{P,Q}\left( \dfrac{4x}{(x+1)^2} \right),
	\]
	where $I_{\varepsilon}=\{i \in [d] : \varepsilon_i=1 \}$ and $R_{\varepsilon}$ is a naturally labeled poset which is obtained from $P_{I_\varepsilon} \oplus Q_{\overline{I_{\varepsilon}}}$ by reordering the label and 
	\[
	f_{P,Q}(x)=\frac{1}{2^{d}}\sum_{\varepsilon \in \{-1,1\}^d}
	W^{(\ell)}_{R_{\varepsilon}}(x)
	\]
	In particular, $h^*(\Cc_{P,Q},x)$ is $\gamma$-positive.
	Moreover, $h^*(\Cc_{P,Q},x)$ is real-rooted if and only if $f_{P,Q}(x)$ is real-rooted.
\end{Theorem}

On the other hand, it is known that, from $h^*(\Cc_{P,Q},x)$, we obtain the $h^*$-polynomials of several non-locally anti-blocking lattice polytopes arising from the posets $P$ and $Q$.
The \textit{order polytope} $\Oc_P$ (\cite{twoposetpolytopes}) of $P$ is the $(0,1)$-polytope defined by
\[
\Oc_P:=\{ \xb \in [0,1]^d : x_i \leq x_j \mbox{ if } i <_P j  \}.
\]
Given two lattice polytopes $\Pc, \Qc \subset \RR^d$,
we define
\[
\Pc*\Qc:={\rm conv} ((\Pc \times \{0\}) \cup (\Qc \times \{1\}) ) \subset \RR^{d+1},
\]
which are called the \textit{Cayley sum} of $\Pc$ and $\Qc$,
and define 
\[
\Omega(\Pc,\Qc):={\rm conv} ((\Pc \times \{1\}) \cup (-\Qc \times \{-1\}) ) \subset \RR^{d+1}.
\]
\begin{Proposition}[{\cite[Theorem 1.1]{HMTgamma}}]
	Let $P$ and $Q$ be two finite posets on $[d]$.
	Then
	one has
	\[h^*(\Cc_{P,Q},x)=h^*(\Gamma(\Oc_P,\Cc_Q),x).\]
	Furthermore, if $P$ and $Q$ has a common linear extension, then we obtain
	\[h^*(\Cc_{P,Q},x)=h^*(\Gamma(\Oc_P,\Oc_Q),x).\]
\end{Proposition}

\begin{Proposition}[{\cite[Theorem 1.4]{HTomega}}]
	Let $P$ and $Q$ be two finite posets on $[d]$.
	Then
	one has
	\[(1+x)h^*(\Cc_{P,Q},x)=h^*(\Omega(\Oc_P,\Cc_Q),x).\]
	Furthermore, if $P$ and $Q$ has a common linear extension, then we obtain
	\[
	(1+x)h^*(\Cc_{P,Q},x)=h^*(\Omega(\Oc_P,\Oc_Q),x).\]
\end{Proposition}

\begin{Proposition}[{\cite[Theorem 4.1]{HOTcayley}}]
	Let $P$ and $Q$ be two finite posets on $[d]$.
	Then
	one has
	\[h^*(\Cc_{P,Q},x)=h^*(\Oc_P * \Cc_Q,x).\]
\end{Proposition}

From these propositions and Theorem \ref{thm:twinnedhpoly}, we obtain the following:
\begin{Corollary}
	Let $P$ and $Q$ be two finite posets on $[d]$.
	Then the $h^*$-polynomials of $\Gamma(\Oc_P,\Cc_Q)$, $\Omega(\Oc_P, \Cc_Q)$,  $\Oc_P*\Cc_Q$  and $\Omega(\Cc_P,\Cc_Q)$ are $\gamma$-positive.
	Furthermore, if $P$ and $Q$ has a common linear extension, then the $h^*$-polynomials of $\Gamma(\Oc_P,\Oc_Q)$ and $\Omega(\Oc_P,\Oc_Q)$ are also $\gamma$-positive.
\end{Corollary}

In the rest of section, we introduce enriched $(P,Q)$-partitions and we show that the Ehrhart polynomial of $\Cc_{P,Q}$ coincides with a counting polynomial of enriched $(P,Q)$-partitions.
Assume that $P$ and $Q$ are naturally labeled.
We say that a map $f : [d] \to \ZZ$ is an \textit{enriched $(P,Q)$-partition}
if, for all $x, y \in [d]$, $f$ satisfies
\begin{itemize}
	\item $x <_P y$, $f(x) \geq 0$ and $f(y) \geq 0 \Rightarrow f(x) \leq f(y)$; 
	\item $x <_Q y$, $f(x) \leq 0$ and $f(y) \leq 0 \Rightarrow f(x) \geq f(y)$.
\end{itemize}
For a map $f : [d] \to \ZZ$, we set $m(f) = \min \{ \{0\} \cup \{ f(x) : x \in 
[d] \} \}$ and $M(f) = \max \{ \{0\} \cup  f(x) : x \in [d]\} \}$.
For each $0 < m \in \ZZ$, let $\Omega_{P,Q}^{(e)}(m)$ denote the number of enriched $(P,Q)$-partitions $f :[d] \to \ZZ$ with $M(f) - m(f) \leq m$.

\begin{Theorem}
	\label{thm:enrichedPQpart}
	Let $P$ and $Q$ be two finite posets on $[d]$. Then one has 	\[
	L_{\Cc_{P,Q}}(m)=\Omega_{P,Q}^{(e)}(m).
	\]
\end{Theorem}

\begin{proof}
	Denote $F(m)$ the set of enriched $(P,Q)$-partitions with $M(f)- m(f) \leq m$.
	We show that there exists a bijection from  $m\Cc_{P,Q} \cap \ZZ^d$ to $F(m)$. 
	
	Take $f \in F(m)$ and set $m(f) = a$ and $M(f)=b$. 
	We set 
	\[
	I=\{i \in [d] : f(i) \geq 0 \}.
	\]
	Let
\[
	x_i=\left\{
	\begin{array}{cl}
	f(i) &  \mbox{ if } i \in I \mbox{ is minimal in } P_I,\\
	\\
	\min \{ f(i) -f(j) : i  \mbox{ covers } j \mbox{ in } P_I\} & \mbox{ if } 
	i \in I \mbox{ is not minimal in } P_I,\\
\\
	-|f(i)| &  \mbox{ if } i \in \overline{I}  \mbox{ is minimal in } Q_{\overline{I}},\\
	\\
	-\min \{ |f(i)| -|f(j)| : i  \mbox{ covers } j \mbox{ in } Q_{\overline{I}}\} & \mbox{ if } 
	i \in \overline{I} \mbox{ is not minimal in } Q_{\overline{I}}.
	\end{array}
	\right.
	\]
	Assume that $I=\{1,\ldots,k \}$ and $\overline{I}=\{k+1,\ldots,d\}$.
	Then we have $(x_1,\ldots,x_k) \in b \Cc_{P_I}$ and $(x_{k+1},\ldots,x_d) \in a \Cc_{Q_{\overline{I}}}$
	by a result of Stanley \cite[Theorem 3.2]{twoposetpolytopes}.
	Hence one obtains $(x_1,\ldots,x_d) \in b \Cc_{P_I} \oplus a \Cc_{Q_{\overline{I}}} \subset m\Cc_{P,Q}$, where $b \Cc_{P_I} \oplus a \Cc_{Q_{\overline{I}}}$ is the free sum of $b \Cc_{P_I}$  and  $a \Cc_{Q_{\overline{I}}}$.
	Similarly, in general, it follows that $(x_1,\ldots,x_d) \in m\Cc_{P,Q}$.
	Therefore, the map $\phi : F(m) \to m \Cc_{P,Q} \cap \ZZ^d$ defined by $\phi(f)=(x_1,\ldots,x_d)$ for each $f \in F(m)$ is well-defined.
	
	Take $(x_1,\ldots,x_d) \in m\Cc_{P,Q} \cap \ZZ^d$.
	We set
	\[
	I= \{ i \in [d] : x_i \geq 0 \}.
	\]
	We define a map $f : [d] \to \ZZ$ by
	\[
	f(i) =
	\left\{
	\begin{array}{cl}
	\max\{x_{j_1} + \dots +  x_{j_k}  :  j_1 <_{P_I} \dots <_{P_I} j_k =i \} & \mbox{ if } i \in I,\\
	\\
	-
	\max\{ |x_{j_1}| + \dots +  |x_{j_k}|  : j_1 <_{Q_{\overline{I}}} \dots <_{Q_{\overline{I}}} j_k =i \} & \mbox{ if } i \in \overline{I}.
	\end{array}
	\right.
	\]
	Assume that $I=\{1,\ldots,k \}$ and $\overline{I}=\{k+1,\ldots,d\}$.
	Then one has $(x_1,\ldots,x_d) \in m(\Cc_{P_I} \oplus (-\Cc_{Q_{\overline{I}}})) \cap \ZZ^d$.
	Moreover, for some integers $a$ and $b$ with $a \leq 0 \leq b$ and $b-a \leq m$, it follows that $(x_1,\ldots,x_k) \in b \Cc_{P_I}$ and $(x_{k+1},\ldots,x_d) \in a \Cc_{Q_{\overline{I}}}$.
We define $f_1: I \to
\ZZ$ by $f_1(i)=f(i)$, and $f_2: \overline{I} \to \ZZ$ by $f_2(i)=-f(i)$.
From \cite[Proof of Theorem 3.2]{twoposetpolytopes},
it follows that $0 \leq f_1(i)  \leq b$ for any $i \in I$ and $f_1(x) \leq f_1(y)$ if $x_{<_{P_I}} y$, and $0 \geq f_2(i) \geq a$ for any $i \in \overline{I}$ and $f_2(x) \leq f_2(y)$ if $x_{<_{Q_{\overline{I}}}} y$. 
Therefore, $f : [d] \to \ZZ$ is an enriched $(P,Q)$-partition with $M(f)-m(f) \leq b - a \leq m$, namely, $f \in F(m)$.
	Similarly, in general, it follows that $f \in F(m)$.
	Thus, the map $\psi:m\Cc_{P,Q} \cap \ZZ^d \to F(m)$ defined by $\psi(\xb)(i)=f(i)$ for each $\xb=(x_1,\ldots,x_d) \in m\Cc_{P,Q} \cap \ZZ^d$ is well-defined. 
	
	Finally, we show that $\phi$ is a bijection. However, this immediately follows by the above and the argument in \cite[Proof of Theorem 3.2]{twoposetpolytopes}.
\end{proof}
Since $\Cc_{P,Q}$ is reflexive, we obtain the following:
\begin{Corollary}
	Let $P$ and $Q$ be two finite naturally labeled posets on $[d]$.
	Then $\Omega^{(e)}_{P,Q}(m)$ is a polynomial in $m$ of degree $d$ and one has
	\[
	\Omega^{(e)}_{P,Q}(m)=(-1)^d\Omega^{(e)}_{P,Q}(-m-1).
	\]
\end{Corollary}

\end{document}